\documentclass[12pt]{article} 

\usepackage{amsmath}
\usepackage{amssymb}
\usepackage{amsthm}
\usepackage[all,cmtip]{xy}
\usepackage{fancyhdr}
\usepackage{euscript}
\usepackage{setspace}
\usepackage{graphicx,epsfig}
\usepackage{palatino, url, multicol}
\usepackage{verbatim}    
\usepackage{mathptmx}

\newcommand{\tab}{\hspace{1pc}}

\newcommand{\R}{\hbox{$\mathbb{R}$}}
\newcommand{\N}{\hbox{$\mathbb{N}$}}

\newcommand{\Q}{\hbox{$\mathbb{Q}$}}

\newcommand{\abs}[1]{\hbox{$\left| {#1} \right|$}}

\newtheoremstyle%
{custom}%
{}
{}
{}
{}
{}
{.}
{ }
{\thmname{}
\thmnumber{}%
\thmnote{\bfseries #3}}%

\newtheoremstyle%
{Theorem}%
{}%
{}%
{\itshape}%
{}%
{}%
{.}%
{ }%
{\thmname{\bfseries #1}%
\thmnumber{\;\bfseries #2}%
\thmnote{\;(\bfseries #3)}}%

\theoremstyle{Theorem}
\newtheorem{thm}{Theorem}
\newtheorem{cor}[thm]{Corollary}
\newtheorem{lemma}[thm]{Lemma}
\newtheorem{prop}[thm]{Proposition}

\theoremstyle{definition}

\theoremstyle{remark}

\theoremstyle{custom}

\begin{document} 

\title{Symmetry in the Sequence of Approximation Coefficients}
 \author{Avraham Bourla\\
   Department of Mathematics\\
   Saint Mary's College of Maryland\\
   Saint Mary's City, MD, 20686\\
   \texttt{abourla@smcm.edu}}
 \date{\today}
 \maketitle

\begin{abstract}
\noindent Let $\{a_n\}_1^\infty$ and $\{\theta_n\}_0^\infty$ be the sequences of partial quotients and approximation coefficients for the continued fraction expansion of an irrational number. We will provide a function $f$ such that $a_{n+1} = f(\theta_{n\pm1},\theta_n)$. In tandem with a formula due to Dajani and Kraaikamp, we will write $\theta_{n \pm 1}$ as a function of $(\theta_{n \mp 1}, \theta_n)$, revealing an elegant symmetry in this classical sequence and allowing for its recovery from a pair of consecutive terms.   
\end{abstract}

\section{Introduction}{}

\noindent Given an irrational number $r$ and a rational number written as the unique quotient $\frac{p}{q}$ of the two integers $p$ and $q$ with $\gcd(p,q)=1$ and $q>0$, our fundamental object of interest from diophantine approximation is the {\bf approximation coefficient} $\theta(r,\frac{p}{q}) := q^2\abs{r-\frac{p}{q}}$. Small approximation coefficients suggest high quality approximations, combining accuracy with simplicity. For instance, the error in approximating $\pi$ using the fraction $\frac{355}{113}=3.14159203539823008849557522124$ is smaller than the error of its decimal expansion to the fifth digit $3.14159 = \frac{314159}{100000}$. Since the former rational also has a much smaller denominator, it is of far greater quality than the latter. Indeed $\theta\big(\pi,\frac{355}{113}\big)< 0.0341$ whereas $\theta\big(\pi,\frac{314159}{100000}\big)> 26535$.\\

\noindent We obtain the high quality approximations for $r$ by using the euclidean algorithm to write $r$ as an infinite continued fraction:
\[r =  a_0 + [a_1,a_2,...] := a_0 + \dfrac{1}{a_1 +\dfrac{1}{a_2 + ...}},\]
where the {\bf partial quotients} $a_0 = a_0(r) \in \mathbb{Z}$ and $a_n = a_n(r) \in \mathbb{N} := \mathbb{Z} \cap [1,\infty)$ for all $n \ge 1$, are uniquely determined by $r$. This expansion also provides us with the infinite sequence of rational numbers 
\[\frac{p_0}{q_0} := \frac{a_0}{1}, \tab \frac{p_n}{q_n} := a_0 + [a_1,...,a_n], \tab n \ge 1,\]
tending to $r$ known as the {\bf convergents} of $r$. Define the approximation coefficient of the n$^{\operatorname{th}}$ convergent of $r$ by 
\[\theta_n := \theta\bigg(r,\frac{p_n}{q_n}\bigg) = q_n^2\abs{r - \frac{p_n}{q_n}}\] 
and refer to the sequence $\{\theta_n\}_0^\infty$ as the {\bf sequence of approximation coefficients}. Since adding an integer to a fraction does not change its denominator, the number $x_0 := r - a_0$ shares the same sequences $\{a_n\}_1^\infty$ and $\{\theta_n\}_0^\infty$ as $r$, allowing us to restrict our attention solely to the unit interval. Throughout this paper, we fix an initial seed $x_0 \in (0,1) -\Q$ and let $\{a_n\}_1^\infty$ and $\{\theta_n\}_0^\infty$ be its sequences of partial quotients and approximation coefficients. While the rest of this section is not a prerequisite, the following results illustrate some of the key properties for this classical sequence and are given for motivation as well as for sake of completeness.\\  

\noindent For all $n \ge 0$, it is well known \cite[Theorem 4.6]{Burger} that $\abs{x_0 - \frac{p_n}{q_n}} < \frac{1}{q_n{q_{n+1}}} < \frac{1}{q_n^2}$. We conclude that $\theta_n < 1$ for all $n \ge 0$. Conversely, Legendre \cite[Theorem 5.12]{Burger} proved that if $\theta(x_0,\frac{p}{q}) < \frac{1}{2}$ then $\frac{p}{q}$ is a convergent of $x_0$. In 1891, Hurwitz proved that there exist infinitely many pairs of integers $p$ and $q$, such that $\theta(x_0,\frac{p}{q}) < \frac{1}{\sqrt{5}} \approx 0.4472$ and that this constant, known as the {\bf Hurwitz Constant}, is sharp. Therefore, all irrational numbers possess infinitely many high quality approximations using rational numbers, whose associated approximation coefficients are less than $\frac{1}{\sqrt{5}}$. Using Legendre's result, we see that all these high quality approximations must belong to the sequence of continued fraction convergents for $x_0$.\\ 

\noindent We may restate Hurwitz's theorem as the sharp inequality $\displaystyle{\liminf_{n\to\infty}}\{\theta_n\}$ $\le \frac{1}{\sqrt{5}}$. In general, we use the value of $\displaystyle{\liminf_{n \to \infty}}\{\theta_n\}$ to measure how well can $x_0$ be approximated by rational numbers. The set of values taken by $\displaystyle{\liminf_{n \to \infty}}\big\{\theta_n(x_0)\big\}$, as $x_0$ varies in the set of all irrational numbers in the interval, is called the {\bf Lagrange Spectrum} and those irrational numbers $x_0$ which construe the spectrum, that is, for which $\displaystyle{\liminf_{n \to \infty}}\big\{\theta_n(x_0)\big\} > 0$ are called {\bf badly approximable numbers}. It is known \cite[Theorem 7.3]{Burger} that $x_0$ is badly approximable if and only its sequence of partial quotients $\{a_n\}_1^\infty$ is bounded. For more details about the Lagrange Spectrum, refer to \cite{CF}.\\  
   
 \noindent In 1895, Vahlen \cite[Corollary 5.1.13]{DK} proved that for all $n \ge 1$ we have the sharp inequality
\begin{equation}\label{Vahlen}
\min\{\theta_{n-1},\theta_n\} < \frac{1}{2}
\end{equation} 
and in 1903, Borel \cite[Theorem 5.1.5]{DK} proved the sharp inequality $\min\{\theta_{n-1},\theta_n,\theta_{n+1}\} < 5^{-.5}$. More recent improvements include the sharp inequalities $\min\big\{\theta_{n-1},\theta_n,\theta_{n+1}\big\} < (a_{n+1}^2 + 4)^{-.5}$ and $\max\big\{\theta_{n-1},\theta_n,\theta_{n+1}\big\} > (a_{n+1}^2 + 4)^{-.5}$, due to Bagemihl and McLaughlin \cite{BM} and Tong \cite{Tong}. Therefore, this sequence exhibits a bounding symmetry on a triple of consecutive terms, which stems from its internal connection with the sequence of partial quotient.\\ 

\noindent For instance, we write
\[\pi - 3 = \dfrac{1}{7 + \dfrac{1}{15+ \dfrac{1}{1+ \dfrac{1}{292 + ...}}}} = [7,15,1,292,1,1,1,2,1,3,1,14,2,...].\]
The first ten convergents $\big\{\frac{p_n}{q_n}\big\}_0^9$ are 
\[\bigg\{\frac{0}{1}, \frac{1}{7},\frac{15}{106},\frac{16}{113}, \frac{4687}{33102}, \frac{4703}{33215}, \frac{9390}{66317}, \frac{14093}{99532}, \frac{37576}{265381}, \frac{51669}{364913} \bigg\}\] 
and the best upper bounds for $\{\theta_n\}_0^9$ using a four digit decimal expansion are 
\[\{0.1416, 0.0612, 0.9351, 0.0034, 0.6237, 0.3641, 0.5363, 0.2885, 0.6045, 0.2134\}.\] 
In particular, the small approximation coefficient $\theta_3=0.0034$ helps explain why the rational number $\frac{355}{113} = 3 + \frac{16}{113}$, first discovered by Archimedes (c. 287BC - c. 212BC), was a popular approximation for $\pi$ throughout antiquity.\\

\section{Preliminary results}

\noindent In 1921, Perron \cite{Perron} proved that
\begin{equation}\label{Perron_classic}
\dfrac{1}{\theta_{n-1}} = [a_{n+1},a_{n+2},...] + a_n + [a_{n-1},a_{n-2},...,a_1], \hspace{1pc} n \ge 1,
\end{equation}
where we take $[\emptyset] := 0$ when $n=1$. Thus, as far as the flow of information goes, the entire sequence of partial quotients is needed in order to generate a single member in the sequence of approximation coefficients. In 1978, Jurkat and Peyerimhoff \cite{JP} showed that for all irrational numbers and for all $n \ge 1$, the point $(\theta_{n-1},\theta_n)$ lies in the interior of the triangle with vertices $(0,0), (0,1)$ and $(1,0)$. As a result, we have 
\begin{equation}\label{max=1}
\theta_{n-1} + \theta_n < 1, 
\end{equation}
which is an improvement of Vahlen's result \eqref{Vahlen}. In addition, they proved that $a_{n+1}$ can be written as a function of $(\theta_{n-1},\theta_n)$ but came short of providing a simple expression, which applies to all cases. Combining this observation with the pair of symmetric identities 
\[\theta_{n+1} = \theta_{n-1} + a_{n+1}\sqrt{1 - 4\theta_{n-1}\theta_n} - a_{n+1}^2\theta_n, \hspace{1pc} n \ge 1\]
and
\[\theta_{n-1} = \theta_{n+1} + a_{n+1}\sqrt{1 - 4\theta_{n+1}\theta_n} - a_{n+1}^2\theta_n, \hspace{1pc} n \ge 1,\]
due to Dajani and Kraaikamp \cite[proposition 5.3.6]{DK}, allows us to recover the tail of the sequence of approximation coefficients from a pair of consecutive terms.\\ 

\noindent We abbreviate the last two equations to the single working formula
\begin{equation}\label{theta_n_pm_1}
\theta_{n\pm1} = \theta_{n\mp1} + a_{n+1}\sqrt{1 - 4\theta_{n\mp1}\theta_n} - a_{n+1}^2\theta_n, \hspace{1pc} n \ge 1. 
\end{equation}
Our goal, obtained in Theorem \ref{result}, is to provide a real valued function $f$ such that $a_{n+1} = f(\theta_{n\pm1},\theta_n)$. This will enable us, as expressed in Corollary \ref{result_cor}, to eliminate $a_{n+1}$ from formula \eqref{theta_n_pm_1} without disrupting its elegant symmetry. This will enable us to recover the {\it entire} sequence $\{\theta_n\}_0^\infty$ from a pair of consecutive terms.\\

\section{Symbolic dynamics}

\noindent The continued fraction expansion is a symbolic representation of irrational numbers in the unit interval as an infinite sequence of positive integers. Let $\lfloor \cdot \rfloor$ be the {\bf floor} function, whose value on a real number $r$ is the largest integer smaller than or equal to $r$. Then we obtain this expansion for the initial seed $x_0 \in (0,1) - \Q$ by using the following infinite iteration process: 
\begin{enumerate}
\item Let $n :=1$.
\item Set  the {\bf reminder} of $x_0$ at time $n$ to be $r_n := \frac{1}{x_{n-1}} \in (1,\infty)$. 
\item Define the {\bf digit} and {\bf future} of $x_0$ at time $n$ to be the integer part and fractional part of $r_n$ respectively, that is, $a_n := \lfloor r_n \rfloor \in \N$ and $x_n := r_n - a_n \in (0,1) - \Q$. Increase $n$ by one and go to step 2. 
\end{enumerate}
Using this iteration scheme, we obtain
\[x_0 = \frac{1}{r_1} = \dfrac{1}{a_1 + x_1} = \dfrac{1}{a_1 + \dfrac{1}{r_2}} = \dfrac{1}{a_1 + \dfrac{1}{a_2 + x_2}} = \dfrac{1}{a_1 + \dfrac{1}{a_2 + \frac{1}{r_3}}} = ... \]
hence, the quantity $a_n$ is no other than the $n^{\operatorname{th}}$ partial quotient of $x_0$. We relabel $a_n$ as the digit for $x_0$ at time $n$ in order to emphasis the underlying dynamical structure at hand and write
\begin{equation}\label{r}
x_0 = [r_1] = [a_1,r_2] = [a_1,a_2,r_3] = ...
\end{equation} 
The quantity $x_n = r_n - a_n$ is the value of $x_{n-1}$ under the {\bf Gauss Map} 
\begin{equation}\label{T}
T: \big((0,1) - \Q\big) \to \big((0,1) - \Q\big), \hspace{1pc} T(x) := \frac{1}{x} - \bigg\lfloor \frac{1}{x} \bigg\rfloor.
\end{equation}
This map is realized as a left shift operator on the set of infinite sequences of digits, i.e.
\[ [a_n,a_{n+1},r_{n+2}] = x_{n-1} \overset{T}\mapsto x_n =  [a_{n+1},r_{n+2}], \hspace{1pc} n \ge 1.\]
We preserve the $n$ digits that the map $T^n$ erases from this symbolic representation of $x_0$ by defining the {\bf past} of $x_0$ at time $n \ge 1$ to be 
\begin{equation}\label{y_n}
y_n := -a_n -  [a_{n-1},a_{n-2},...,a_1] < -1.
\end{equation} 
The {\bf natural extension map} 
\[\mathcal{T}(x,y) := \bigg(\frac{1}{x} - \bigg\lfloor \frac{1}{x} \bigg\rfloor, \tab \frac{1}{y} - \bigg\lfloor \frac{1}{x} \bigg\rfloor \bigg) = \bigg(T(x), \tab \frac{1}{y} - \bigg\lfloor \frac{1}{x} \bigg\rfloor \bigg),\]
is well defined whenever $x$ is an irrational number and $y<-1$, providing us with the relationship
\[(x_{n+1},y_{n+1}) = \mathcal{T}(x_n,y_n), \hspace{1pc} n \ge 1.\]
Since $x_n$ is uniquely determined by $\{a_n\}_{n+1}^\infty$ and $y_n$ is uniquely determined by $\{a_n\}_1^n$, this map can be thought of as one tick of the clock in the symbolic representation of $x_0$ using the sequence $\{a_n\}_1^\infty$
\[ [[a_1,a_2,...,a_n | a_{n+1}, a_{n+2} ...]] \overset{\mathcal{T}}\mapsto [[a_1,a_2,...,a_n, a_{n+1} | a_{n+2}, ...]], \]  
advancing the present time denoted by $|$ one step into the future. 

\section{Dynamic pairs vs. Jager Pairs}

\noindent Using the dynamical terminology of the last section, we restate Perron's result \eqref{Perron_classic} as
\begin{equation}\label{Perron}
\theta_{n-1} = \frac{1}{x_n - y_n}, \hspace{1pc} n \ge 1. 
\end{equation}
Define the region $\Omega := (0,1) \times (-\infty,-1) \subset \R^2$ and the map
\begin{equation}\label{Psi} 
\Psi: \Omega \to \R^2, \hspace{1pc} \Psi(x,y) := \bigg(\dfrac{1}{x-y}, -\dfrac{x{y}}{x-y}\bigg),
\end{equation}
which is clearly well-defined and continuous. Since for all $n \ge 1$, we have $(x_n,y_n) \in \Omega$, we use formulas \eqref{T} and \eqref{Perron} to obtain  
 \[\dfrac{1}{\theta_n} = x_{n+1} - y_{n+1} = \bigg(\dfrac{1}{x_n} -a_{n+1}\bigg)- \bigg(\dfrac{1}{y_n} -a_{n+1}\bigg) = -\dfrac{x_n-y_n}{x_n{y_n}}, \hspace{2pc} n \ge 0,\] 
so that
\begin{equation}\label{Psi_theta} 
\Psi(x_n,y_n) = (\theta_{n-1}, \theta_n).
\end{equation}
We call $(\theta_{n-1}, \theta_n)$ the {\bf Jagger Pair} of $x_0$ at time $n$. We also denote the image $\Psi(\Omega)$ by $\Gamma$. Then
\begin{prop}\label{triangle}
The set $\Gamma$ is the open region interior to the triangle in $\R^2$ with vertices $(0,0), (1,0)$ and $(1,0)$.
\end{prop}
\begin{proof}
For every positive integer $k \ge 2$, define the open region $\Omega_k := (\frac{1}{k},1) \times (-k,-1)$, whose boundary contains the open line segments $(\frac{1}{k},1) \times \{-1\}, \{1\} \times (-k, -1), \left(\frac{1}{k},1\right) \times \{-k\}$ and $\{\frac{1}{k}\} \times (-k,-1)$. Since $\Psi$ is continuous, $\Gamma_k := \Psi(\Omega_k)$ is the open region interior to the image of the boundary for $\Omega_k$ under $\Psi$, which will now find explicitly.\\

\noindent From the definition \eqref{Psi} of $\Psi$, we have 
\begin{equation}\label{x=uy}
x = \frac{1}{u} + y,
\end{equation}
and
\begin{equation}\label{v=uxy}
v =  -\frac{x{y}}{x-y} = -u{x}y.
\end{equation}
Set $y :=-1$ and $x \in \left(\frac{1}{k},1\right)$ so that, by the definition \eqref{Psi} of $\Psi$, we have $u = \frac{1}{x-y} = \frac{1}{x+1} \in \left(\frac{1}{2},\frac{k}{k+1}\right)$. Formulas \eqref{x=uy} and \eqref{v=uxy} now yield $v = {u}\left(\frac{1}{u} - 1  \right) = 1 - u \in \left(\frac{1}{2},\frac{1}{k+1} \right)$. Conclude that $\Psi$ maps the open line segment $(\frac{1}{k},1) \times \{-1\}$ in the $x{y}$-plane to the open line segment between the points $\left(\frac{1}{2},\frac{1}{2}\right)$ and $\left(\frac{k}{k+1},\frac{1}{k+1}\right)$ in the $u{v}$-plane.\\   

\noindent Set $x := 1$ and $y \in (-k,-1)$ so that, by definition \eqref{Psi} of $\Psi$, we have $u = \frac{1}{1-y} \in \left(\frac{1}{2}, \frac{1}{k+1}\right)$. Formulas \eqref{x=uy} and \eqref{v=uxy} now yield $v = {u}\left(\frac{1}{u} - 1  \right) = 1 - u \in \left(\frac{1}{2},\frac{k}{k+1}\right)$. Conclude that $\Psi$ maps the open line segment $\{1\} \times (-k,-1)$ in the $x{y}$-plane to the open line segment between the points $\left(\frac{1}{2},\frac{1}{2}\right)$ and $\left(\frac{1}{k+1},\frac{k}{k+1}\right)$ in the $u{v}$-plane.\\  

\noindent Set $y :=-k$  and $x \in \left(\frac{1}{k}, 1 \right)$, so that, by the definition \eqref{Psi} of $\Psi$, we have $u = \frac{1}{x-y} = \frac{1}{x+k} \in \left(\frac{1}{k+1}, \frac{k}{k^2+1}\right)$. Formulas \eqref{x=uy} and \eqref{v=uxy} now yield $v = -u\left(\frac{1}{u} - k\right)(-k)= k - k^2{u} \in \left(\frac{k}{k^2+1},\frac{k}{k+1}\right)$. Conclude that $\Psi$ maps the open line segment $(0,1) \times \{-k\}$ $x{y}$-plane to the open line segment between the points $\left(\frac{1}{k+1},\frac{k}{k+1}\right)$ and $\left(\frac{k}{k^2+1},\frac{k}{k^2+1}\right)$ in the $u{v}$-plane.\\  

\noindent Finally, set $x :=\frac{1}{k}$ and $y \in (-k,-1)$, so that, by the definition \eqref{Psi} of $\Psi$, we have $u \in \left(\frac{k}{k^2+1}, \frac{k}{k+1}\right)$. Formulas \eqref{x=uy} and \eqref{v=uxy} now yield $v = \frac{u}{k}\left(\frac{1}{u} - \frac{1}{k}\right) =\frac{1}{k} - \frac{u}{k^2} \in \left(\frac{k}{k^2+1},\frac{k}{k+1}\right)$. Conclude that $\Psi$ maps the open line segment $\big\{\frac{1}{k}\big\} \times (-k,-1)$ in the $x{y}$-plane to the open line segment between the points $\big(\frac{k}{k^2+1},\frac{k}{k^2+1}\big)$ and $\big(\frac{k}{k+1},\frac{1}{k+1}\big)$ in the $u{v}$-plane. From the continuity of $\Psi$, we have  
\[\Gamma = \Psi(\Omega) = \Psi\left(\bigcup_2^\infty\Omega_k\right)= \bigcup_2^\infty\Psi(\Omega_k) = \bigcup_2^\infty\Gamma_k.\] 
Therefore, we conclude the desired result after letting $k \to \infty$.\\ 

\end{proof}

\noindent Note that since $(\theta_{n-1},\theta_n)=\Psi(x_n,y_n) \in \Gamma$, this observation is in accordance with formula \eqref{max=1}. 
\begin{lemma}
The map $\Psi:\Omega \to \Gamma$ is a homeomorphism with inverse:
\begin{equation}\label{Psi_inv}
\Psi^{-1}(u,v) := \bigg(\dfrac{1-\sqrt{1-4{u}v}}{2u}, -\dfrac{1+ \sqrt{1-4{u}v}}{2u}\bigg),  
\end{equation}
\end{lemma}
\begin{proof} 
First, we will show that $\Psi$ is a bijection. Since the map $\Psi$ is surjective onto its image $\Gamma$, we need only show injectiveness. Let $(x_1,y_1), (x_2,y_2)$ be two points in $\Omega$ such that
\[ \bigg(\dfrac{1}{x_1-y_1}, -\dfrac{x_1{y_1}}{(x_1-y_1)}\bigg) = \Psi(x_1,y_1) = \Psi(x_2,y_2) =  \bigg(\dfrac{1}{x_2-y_2}, -\dfrac{x_2{y_2}}{(x_2-y_2)}\bigg).\]  
By equating the first and then the second components of the exterior terms, we obtain that
\begin{equation}\label{x-y}
x_1 - y_1 = x_2 - y_2 
\end{equation}
and then, that $x_1{y_1} = x_2{y_2}$. Therefore,
\[(x_1 +y_1)^2 = (x_1 - y_1)^2 + 4x_1y_1 = (x_2 - y_2)^2 + 4x_2y_2 = (x_2 + y_2)^2.\]
Since both these points are in $\Omega$ they must lie below the line $x + y = 0$, hence $x_1 + y_1 = x_2 + y_2 < 0$. Another application of condition \eqref{x-y} now proves that $x_1=x_2$ and $y_1=y_2$, hence $\Psi$ is injective.\\

\noindent Since both $\Psi$ and $\Psi^{-1}$ are clearly continuous, it is left to prove that $\Psi^{-1}$ is well-defined and that it is the inverse for $\Psi$. Given $(u_0,v_0) \in \Gamma$, set 
\[(x_0,y_0) := \Psi^{-1}(u_0,v_0) = \bigg(\dfrac{1-\sqrt{1-4{u_0}v_0}}{2u_0}, -\dfrac{1+ \sqrt{1-4{u_0}v_0}}{2u_0}\bigg).\] 
From proposition \ref{triangle}, we know that $\Gamma$ lies entirely underneath the line $u+v=1$ in the $u{v}$ plane. The only point of intersection for this line and the hyperbola $4{u}v=1$ is the point $(u,v) = \big(\frac{1}{2},\frac{1}{2}\big)$, hence $\Gamma$ must lie underneath this hyperbola as well. We conclude $4{u_0}v_0 < 1$, so that both $x_0$ and $y_0$ must be real. Another implication of the inequality $u+v < 1$ is that $4{u}v < 4u - 4u^2$, hence $1 - 4u{v} >  4u^2 - 4u + 1 = (2u-1)^2$. Conclude that $1 + \sqrt{1 - 4u{v}} > 2u$ so we must have $y_0 = - \frac{1+\sqrt{1 - 4u_0{v_0}}}{2u_0} < -1$.\\

 \noindent To prove that $x_0\in (0,1)$, we first observe that $\sqrt{1-4{u_0}v_0} < 1$ implies that $1 - \sqrt{1 - 4u_0{v_0}} > 0$, hence $x_0$ is positive. If we further assume by contradiction that $x_0 \ge 1$, then the definition of $\Psi^{-1}$ \eqref{Psi_inv} will imply the inequality
\[1+\sqrt{1 - 4u_0{v_0}} \le x_0\left(1+\sqrt{1 - 4u_0{v_0}}\right) = \frac{1}{2u_0}\left(1-\sqrt{1 - 4u_0{v_0}}\right)\left(1+\sqrt{1 - 4u_0{v_0}}\right) = 2v_0\] 
so that we obtain the inequality
\[4v_0^2 - 4v_0 + 1 = (2v_0-1)^2 \ge 1-4u_0{v_0}.\] 
After the appropriate cancellations and rearrangements, we obtain the inequality $u_0 +v_0 \ge 1$, which is in contradiction to proposition \ref{triangle}. Conclude that $(x_0,y_0) \in \Omega$ and $\Psi^{-1}:\Gamma \to \Omega$ is well-defined.


 
\noindent Finally, we will show that $\Psi^{-1}$ is the inverse for $\Psi$. Let $(u,v) \in \Gamma$ and set $(x,y) := \Psi^{-1}(u,v) \in \Omega$. Using the definitions \eqref{Psi} and \eqref{Psi_inv} of $\Psi$ and $\Psi^{-1}$, the first component of $\Psi(x,y)$ is 
\[\dfrac{1}{x-y} = \bigg(\frac{1 - \sqrt{1 - 4u{v}}}{2u} + \dfrac{1 + \sqrt{1 - 4u{v}}}{2u}\bigg)^{-1} = \left(\dfrac{2}{2u}\right)^{-1} = u\] and its second component is 
\[-\dfrac{x{y}}{(x-y)} = -{u}(x{y}) = u\bigg(\dfrac{1}{4u^2}\big(1 - \sqrt{1 - 4{u}v}^2\big)\bigg) = \dfrac{1}{4u}\cdot4u{v} = v,\]
hence $\Psi^{-1}$ is the right inverse for $\Psi$. Since $\Psi$ is a bijection, we conclude it is the (two-sided) inverse for $\Psi$, completing the proof.

\end{proof}

\section{Result}

\begin{thm}\label{result}
Let $x_0$ be an irrational number in the unit interval and let $n \in \N$. If $a_{n+1}$ is the digit at time $n+1$ in the continued fraction expansion for $x_0$ and if $(\theta_{n-1},\theta_n,\theta_{n+1})$ are the approximation coefficients for $x_0$ at time $n-1,n$ and $n+1$, then
\begin{equation}\label{digit}
a_{n+1} =  \bigg\lfloor \frac{1+ \sqrt{1-4\theta_{n-1}\theta_n}}{2\theta_n} \bigg\rfloor =\bigg\lfloor \frac{1+ \sqrt{1-4\theta_{n+1}\theta_n}}{2\theta_n} \bigg\rfloor.
\end{equation} 
\end{thm}
\begin{proof}
Let $(x_n,y_n)$ be the dynamic pair of $x_0$ at time $n$. Formula \eqref{Psi_theta}, the fact that $\Psi$ is a homeomorphism and the definition \eqref{Psi_inv} of $\Psi^{-1}$ yield
\begin{equation}\label{present_digits}
(x_n,y_n) =  \Psi^{-1}(\theta_{n-1},\theta_n) = \bigg(\dfrac{1-\sqrt{1 - 4\theta_{n-1}\theta_n}}{2\theta_{n-1}}, -\dfrac{1+\sqrt{1 - 4\theta_{n-1}\theta_n}}{2\theta_{n-1}}\bigg).
\end{equation}
Using formula \eqref{r}, we write $x_n = [a_{n+1},r_{n+2}] = \frac{1}{a_{n+1}+[r_{n+2}]}$, so that the first components in the exterior terms of formula \eqref{present_digits} equate to 
\[a_{n+1} + [r_{n+2}] = \frac{2\theta_{n-1}}{1-\sqrt{1 - 4\theta_{n-1}\theta_n}} = \frac{1+ \sqrt{1-4\theta_{n-1}\theta_n}}{2\theta_n}.\] 
But since $[r_{n+2}] = x_{n+1} <1$, we have 
\[a_{n+1} = \big\lfloor a_{n+1} + [r_{n+2}] \big\rfloor = \bigg\lfloor \frac{1+ \sqrt{1-4\theta_{n-1}\theta_n}}{2\theta_n} \bigg\rfloor,\] 
which is the first equality in the equations \eqref{digit}.\\

\noindent Next, we equate the second components in the exterior terms of formula \eqref{present_digits},which, after using formula \eqref{y_n}, yields 
\[a_n + [a_{n-1},...,a_1] = \frac{1+\sqrt{1 - 4\theta_{n-1}\theta_n}}{2\theta_{n-1}}.\] 
But since $[a_{n-1},...,a_1] <1$, we conclude
\[a_n = \big\lfloor a_n + [a_{n-1},...,a_1] \big\rfloor = \bigg\lfloor  \dfrac{1+\sqrt{1 - 4\theta_{n-1}\theta_n}}{2\theta_{n-1}} \bigg\rfloor.\] 
Adding one to all indices establishes the equality of the exterior terms in the equations \eqref{digit} and completes the proof. 
\end{proof}

\noindent As a direct consequence of this theorem and formula \eqref{theta_n_pm_1}, we obtain:
\begin{cor}\label{result_cor}
Assuming the hypothesis of the theorem, we have
\[\theta_{n \pm 1} = \theta_{n \mp 1} +  \bigg\lfloor \frac{1 + \sqrt{1-4{\theta_{n \mp 1}}\theta_n}}{2\theta_n} \bigg\rfloor\sqrt{1 - 4{\theta_{n \mp 1}}\theta_n} -  \bigg\lfloor \frac{1 + \sqrt{1-4{\theta_{n \mp 1}}\theta_n}}{2\theta_n} \bigg\rfloor^2\theta_n.\]
\end{cor}

\section{Acknowledgments}

This paper is a development of part of the author's Ph.D. dissertation at the University of Connecticut. Benefiting tremendously from the patience and rigor his advisor Andrew Haas, he would like to thank him for all his efforts. In addition, he would like to extend his gratitude to Alvaro Lozano-Robledo for his suggestions and corrections.


\end{document}